\documentclass{amsart}[12pt]
\usepackage{amsmath,amssymb,amsfonts,amsthm}

\def\pa{\partial}

\newtheorem{theorem}{Theorem}

\title{Quadratic equations and monodromy evolving
deformations}
\author{Yousuke Ohyama}
\address{Graduate School of Information Science and Technology, 
         Osaka University, Toyonaka, Osaka 560-0043, Japan}
\email{ohyama@math.sci.osaka-u.ac.jp}

\begin{document}
\maketitle

\section{Introduction} 
In this paper we study a special class of monodromy evolving deformations (MED),
which represents Halphen's quadratic system.  

In 1996, Chakravarty and  Ablowitz \cite{CA} showed that 
 a fifth-order equation (DH-V) 
 \begin{eqnarray}
 \omega_1^{\prime} &=&\omega_2\omega_3-\omega_1(\omega_2 + \omega_3)+ \phi^2, \nonumber\\
 \omega_2^{\prime} &=&\omega_3\omega_1-\omega_2(\omega_3 + \omega_1)+ \theta^2,\nonumber \\
 \omega_2^{\prime} &=&\omega_3\omega_1-\omega_2(\omega_3 + \omega_1)- \phi\theta, \label{fifth} \\
\phi^{\prime} &=& \ \ \omega_1(\theta- \phi)-\omega_3(\theta + \phi), \nonumber \\
\theta^{\prime} &=& - \omega_2(\theta- \phi)-\omega_3(\theta + \phi), \nonumber 
\end{eqnarray}
which arises in complex Bianchi IX cosmological models
can be represented by MED. The DH-V is solved by the Schwarzian function
$S(z; 0,0,a)$ (three angles of the Schwarzian triangle are $0, 0$ and $a\pi$)
and a special case of Halphen's quadratic system. 
Since generic Schwarzian functions have natural boundary or moving branch points, 
\eqref{fifth} cannot be obtained as monodromy preserving deformations, because 
 monodromy preserving deformations has the Painlev\'e property.

The system \eqref{fifth} can be represented as the compatibility condition for 
 \begin{eqnarray}
 \frac {\pa Y}{\pa x} &=&  \frac {\mu I -(C_+ x^2 +2D x + C_-)} P  Y,  \label{5med1} \\
 2\frac {\pa Y}{\pa t} &=& [ \nu -(C_+x + D  )] Y -Q(x) \frac {\pa Y}{\pa x}. \label{5med2}
 \end{eqnarray}
Here 
\begin{align*}
P = \alpha_+   x^4 +(\beta_+ + \beta_-)x^2 + \alpha_-, \quad 
Q = \alpha_+ x^3 + \beta_+ x, \\
C_{\pm} =(i\omega_1 \pm \phi)\sigma_1  \pm (\omega_2 \pm i \theta)\sigma_2, \quad  
D =- \omega_3\sigma_3, 
\end{align*}
for $\alpha_{\pm} =(\omega_1-\omega_2)\mp( \theta+ \phi),
\beta_{\pm} =(\omega_1 + \omega_2-2\omega_3) \pm i(\theta- \phi)$. 
$\mu$ is a constant parameter and 
\begin{equation}\label{nu}
\frac{\pa\nu}{\pa x} = \frac{(\beta_-+ 4\omega_3) - \alpha_+ x^2}P \mu. 
\end{equation}
Here the standard Pauli spin matrices $\sigma_j$'s are 
$$
\sigma_1= \left(\begin{array}{cc}
0 & 1 \\
1 & 0
\end{array}
\right), 
\quad 
\sigma_2= \left(\begin{array}{cc}
0 & -i \\
i & 0
\end{array}
\right), 
\quad 
\sigma_3= \left(\begin{array}{cc}
1 & 0 \\
0 & -1
\end{array}
\right). 
$$
Since $\nu$ is not a rational function on $x$,  \eqref{5med2}  does 
not give a monodromy preserving deformation of \eqref{5med1}. 

The aim of this paper gives a basic theory of monodromy evolving deformations
(section ref{sec:mev}). And we describe Halphen's equation related to 
general Schwarzian function $S(z; a,b,c)$ as monodromy evolving deformations.

Halphen studied two types of quadratic equations. His first equation in \cite{H1}
\begin{eqnarray}
 X^{\prime}+Y^{\prime}&=& 2X Y, \nonumber\\
 Y^{\prime}+Z^{\prime}&=& 2Y Z, \label{hal1}\\
 Z^{\prime}+X^{\prime}&=& 2Z X, \nonumber
\end{eqnarray}
is very famous and is appeared in many mathematical fields. 
It is a reduction from the Bianchi IX cosmological models 
or the self-dual Yang-Mills equation \cite{CAC} \cite{CAT} and gives a special
self-dual Einstein metric  \cite{AH}. 
If we set $ y = 2(X + Y + Z)$, $y$ satisfies Chazy's equation 
\begin{equation}\label{Chazy}
y^{\prime\prime\prime}=2y y^{\prime\prime}-3(y^\prime)^2.
\end{equation}
Chazy's equation appeared in his classification of the third order Painlev\'e type 
equation \cite{Chazy}, but \eqref{Chazy} does not have the Painlev\'e property 
because generic solutions has natural boundary. 

Halphen's second equation \cite{H2}
\begin{eqnarray}
 x_1^{\prime}&=& x_1^2 +a (x_1-x_2)^2 +b(x_2-x_3)^2+c(x_3-x_1)^2,\nonumber\\
 x_2^{\prime}&=& x_2^2 +a (x_1-x_2)^2 +b(x_2-x_3)^2+c(x_3-x_1)^2,  \label{hal2}\\
 x_3^{\prime}&=& x_3^2 +a (x_1-x_2)^2 +b(x_2-x_3)^2+c(x_3-x_1)^2, \nonumber
\end{eqnarray}
is less familiar but it is a more general system than \eqref{hal1}. 
Here we use a different form the original Halphen's equation (See \cite{O2}).
In case $a = b = c = -\frac18$, 
\eqref{hal2} is equivalent to \eqref{hal1} by the transform 
$2X = x_2 + x_3, 2Y = x_3 + x_1, 2Z = x_1 + x_2$. 

Halphen's first equation \eqref{hal1} can be solved by theta constants \cite{H1} \cite{O1}.
His second equation \eqref{hal2} can be solved by the Gauss hypergeometric function \cite{O2}.
When $a = b = c = -\frac18$, \eqref{hal2} is solved by $F(1/2, 1/2,1; z)$ which 
is related to $\theta_3(0,\tau)$. 

Halphen's second equation is also a reduction of the self-dual 
Yang-Mills equation  or the Einstein self-dual equation \cite{ACH}. 
Since Halphen's equation or Chazy's equation does not have the Painlev\'e property, 
they are not described as monodromy preserving deformations. 

\par\bigskip

The aim of this paper is to represent the Halphen's second system as monodromy 
evolving deformations (section  \ref{sec:hal2}). Since Halphen's equation do not have the Painlev\'e property, 
it is never represented by monodromy preserving deformations. 

In \cite{CAC} \cite{CAT}, they obtained the Lax pair of Halphen's first equation and Chazy's 
equation, which are special cases of Halphen's second equation. And our Lax pair 
is different from their results even when Halphen's first equation since ours are 
monodromy evolving deformations but \cite{CAC} \cite{CAT} gave the Lax pair as a reduction of 
the self-dual Yang-Mills equation. 

Both \cite{CA} and the author treat special cases of monodromy evolving deformations. 
In section \ref{sec:mev}, we give a general frame of MED only when 
the scalar part of local exponent matrices remains constant.
We do not have general theory of monodromy evolving deformation. But 
it is sufficient to treat such a special case to study Chazy's equation and 
Halphen's equations, which does not have the Painlev\'e property.

This work was paritally done during my stay at the Mathematical Institute, 
University of Oxford. The author would like to express thanks to Prof. Lionel 
Mason and Prof. Nick Woodhouse.

\section{Reduction of the self-dual Yang-Mills equation }
In this section, we survey recent works by the Chakravarty-Ablowitz group.

We take a $\frak{g}$-valued 1-form 
$$A=\sum_{j=1}^4 A_j (x)dx_j,$$
where $\frak{g}$ is a Lie algebra and $x =(x_1, x_2, x_3, x_4) \in \mathbb{R}^4$. 
The curvature 2-form $F = \sum_{j<k} F_{jk}dx_j \wedge dx_k$ is 
given by 
$$F_{jk} = \partial_jA_k - \partial_kA_j  - [A_j ,A_k],$$
where $\partial_j  = \partial/\partial x_j$. The SDYM equation is 
\begin{equation}
F_{12} = F_{34}, \quad F_{13} = F_{42}, \quad F_{14} = F_{23}.  \label{sdym}
\end{equation}

If the $A_j$'s are independent of $x_2, x_3, x_4$ and we  take a special gauge 
such that $A_1$=0, then \eqref{sdym} is reduced the Nahm equations 
\begin{eqnarray}
\partial_tA_2&=& [A_3, A_4], \nonumber\\
\partial_tA_3&=& [A_4, A_2],  \label{nahm}\\
\partial_tA_4&=& [A_2, A_3]. \nonumber
\end{eqnarray}
Here we set $t=x_1$. 

We take $\frak{diff} (S^3)$, the infinite-dimensional Lie algebra of vector fields on $S^3$ as 
the Lie algebra g above. Let $X_1, X_2$ and $X_3$ are divergence-free vector fields on $S^3$ and satisfy commutation relations 
$$[X_j, X_k]=\sum_l \varepsilon_{jkl}X_l, $$
where $\varepsilon_{jkl}$ is the standard anti-symmetric form with 
$\varepsilon_{123} = 1$. Let $O_{jk}\in  SO(3)$ be a matrix such that 
\begin{align*}
\sum_{j,k,l} \varepsilon_{jkl} O_{jp}O_{kq}O_{lr} =\varepsilon_{pqr}, \\
X_j(O_{lk})= \sum_{p} \varepsilon_{jkp}O_{lp}. 
\end{align*}
Then we choose the connection of the form 
$$A_l=\sum_{j,k=1}^3 O_{lj}M_{jk}(t)X_k.$$

Then the 3$\times$3 matrix valued function $M = M(t)$ satisfies the ninth-order Darboux-
Halphen (DH-IX) system \cite{CAT}
\begin{equation}
\frac{dM}{dt} = (\textrm{adj}\, M)T + M^T M - (\textrm{Tr} M)M. \label{dh9}
\end{equation}
Here we set $\textrm{adj}\, M:= \det M\cdot M^{-1}$, and $M^T$ is the transpose of $M$. 
The DH-IX system \eqref{dh9} was also derived by Hitchin \cite{Hitchin} where
it  represents an $SU(2)$-invariant hypercomplex four-manifold. Since the Weyl curvature
of a hypercomplex four-manifold is self-dual, \eqref{dh9} gives a class of self-dual Weyl
Bianchi IX space-times.

\begin{theorem} (1)  
When $M_{jk}(t)= \omega_k(t)\delta_{jk}$, the DH-IX system is equaivalent to 
the Halphen's first equation  \cite{CAC} \cite{CAT}. 

(2)  
When 
$$M=\begin{pmatrix} 
\omega_1 & \theta & 0 \\ 
\phi & \omega_2 &  0 \\ 
0 & 0 & \omega_3 \\ 
\end{pmatrix},
$$
the DH-IX system is the DH-V \eqref{fifth}, which is described by monodromy evolving 
deformations as (\ref{5med1}-\ref{5med2}) \cite{CA}.

(3)  In generic case, the DH-IX system is 
equivalent to Halphen's second equation \cite{ACH}.
\end{theorem}

We explain (3). 
We decompose the matrix $M = M_s+M_a$, where $M_s$  is the symmetric part of $M$ 
and $M_a$  is the anti-symmetric part of $M$.  
We assume that the eigenvalues of the symmetric part $M_s$ of $M$ are distinct.
Then  $M_s$ can be diagonalized using a  complex  orthogonal matrix $P$ and 
we can  write
$$M_s= P d P^{-1}, \quad M_a= P a P^{-1},$$ 
where $d=\textrm{diag}\, (\omega_1, \omega_2, \omega_3).$  
The matrix element of a skew-symmetric matrix
$a$ are denoted as $a_{12}=-a_{21}=\tau_3, a_{23}=-a_{32}=\tau_1, a_{31}=-a_{13}=\tau_2$.  
In  show that equation \eqref{dh9} can be 
reduced to the third-order system 
\begin{eqnarray}
\omega_1^{\prime} &=&  \omega_2\omega_3 - \omega_1(\omega_2 + \omega_3)+ \tau^2, \nonumber\\
\omega_2^{\prime} &=&  \omega_3\omega_1 - \omega_2(\omega_3 + \omega_1)+ \tau^2,\label{ACH}\\ 
\omega_2^{\prime} &=&  \omega_3\omega_1 - \omega_2(\omega_3 + \omega_1)+ \tau^2.\nonumber
\end{eqnarray}
where 
$$
\tau^2 = \alpha^2_1(\omega_1-\omega_2)(\omega_3-\omega_1)+
 \alpha^2_2(\omega_2-\omega_3)(\omega_1-\omega_2)+ 
\alpha^2_3(\omega_3-\omega_1)(\omega_2-\omega_3).
$$
The system \eqref{ACH} is equivalent to Halphen's second equation \eqref{hal2} by 
$$
2\omega_1 = -x_2 - x_3, \quad
2\omega_2 = -x_3 - x_1, \quad 
2\omega_3 = -x_1 - x_2, 
$$
and 
$$
8a = \alpha^2_1 + \alpha^2_2- \alpha^2_3 - 1, \quad
8b = -\alpha^2_1 + \alpha^2_2 + \alpha^2_3 - 1, \quad
8c = \alpha^2_1 - \alpha^2_2 + \alpha^2_3 -1. $$
The system \eqref{ACH} can be solved by the Schwarzian function $S(x; \alpha_1, 
\alpha_2, \alpha_3)$.

\section{Halphen's  equation}
In this section we review Halphen's equation.  Halphen's first equation \eqref{hal1} 
can be solved by theta constants. See \cite{O1}.
\begin{eqnarray*}
X=  2\frac \partial {\partial t}\log \left[ \theta_2\left(0,\frac {a t+b}{c t+d}\right)
(c t+d)^{-1/2}\right]\\
Y=  2\frac \partial {\partial t}\log \left[ \theta_3\left(0,\frac {a t+b}{c t+d}\right)
(c t+d)^{-1/2}\right],\\
Z=  2\frac \partial {\partial t}\log \left[ \theta_4\left(0,\frac {a t+b}{c t+d}\right)
(c t+d)^{-1/2}\right].
\end{eqnarray*}
Here $\theta_j(z,\tau)$ is Jacobi's theta function
and $ad-bc=1$.  When $j=2,3,4$, $\theta_j(z,\tau)$ 
is an even function as $z$.  Since generic solutions of \eqref{hal1} have 
natural boundary, \eqref{hal1} does not have the Painlev\'e property. 

If we set $y=2(X+Y+Z)$, $y$ satisfies Chazy's equation \eqref{Chazy}, 
which is solved by
 $$y(t)=4\frac \partial {\partial t}\log \left[\vartheta_1^{\prime}
\left(0,\frac {at+b}{ct+d}\right)(ct+d)^{-3/2}\right],$$
where $ad-bc=1$. 

Halphen's first equation and Chazy's equation are special cases of Halphen's second 
equation \eqref{hal2}, which is solved by hypergeometric functions. 
For details, see \cite{O2}.
We take a Fuchsian equation
$$\frac {d^2 y}{dz^2}=\left( \frac{a+b}{z^2}+ \frac{c+b} {(z-1)^2}- \frac{2b}{z(z-1)}
\right)y.$$
Let $t$ be a ratio of two solutions of the above equation. We set 
$$x_1=\frac{d}{d t}\log y, \quad x_2=\frac{d}{d t}\log \frac{y}{z},
\quad x_3=\frac{d}{d t}\log \frac{y}{z-1}.$$
Then $x_1, x_2$ and $x_3$ satisfy \eqref{hal2}.

Chazy's equation \eqref{Chazy} can be solved by the hypergeometric equation
$$x(1-x)  \frac{d^2y}{   dx^2} + \left(\frac12 -\frac76x\right) \frac{dy}{ dx}
 -   \frac 1{144}  y=0,$$ 
and   is a special case of Halphen's second equation \eqref{hal2} when 
 $$a=-\frac{31}{288}, \qquad b=-\frac{23}{288}, \qquad c=-\frac{41}{288}.$$

\section{Monodromy evolving deformations}\label{sec:mev}
We give a basic theory of monodromy evolving deformation, which is a generalization 
of monodromy preserving deformation by Schlesinger \cite{Sch05}.
We study the special case of evolving. We assume that 
the scalar part of local exponent matrices will change. 

The solution of a linear equation
\[
\frac {dY}{d x}  =  \sum^{n}_{k=1}  \frac {A_k} {x-x_j} Y,  \\
\]
can be developed as 
\begin{align*}
 Y(x)\sim & Y_j(x) (x-x_j)^{L_j}, \qquad  Y_j(x)=O((x-x_j)^0), \\ 
 Y(x)\sim & Y_\infty(x) x^{-L_\infty}, \qquad  Y_\infty(x)=I+ O(1/x),
\end{align*}
around the singular points $x=x_1, x_2, ..., x_n$ and $\infty$. 
Here    $I$ is the unit matrix.

When  the monodromy is preserved, $Y$ satisfies the deformation equation
\[
\frac{\pa Y}{\pa x_j}  =-\frac{A_j}{x-x_j}Y.
\]
And the compatibility condition gives the Schlesinger equation
\[
\frac{\pa A_k}{\pa x_j}  = \frac{ [A_k, A_j] }{x_k-x_j}, \qquad j\not=k.
\]

We will study monodromy evolving deformations when $L_j$ will evolve according to
\begin{equation}\label{evolv}
\frac {\pa L_j}{\pa x_k} =  f_{jk} I
\end{equation} 
for any $j,k =1, 2, ..., n$. Here $f_{jk}=f_{jk}(t)$ is a scalar function. 
The local exponent $L_\infty$ evolve as
$$ \frac {\pa L_\infty}{\pa x_k} =  -\sum_{j=1}^n  f_{jk} I,$$
because the sum of eigenvalues of all local exponents is invariant. 

\begin{theorem}
If the local exponents $L_j$ evolve as \eqref{evolv}, $Y$ satisfies the deformation
equation
\[ 
\frac{\partial Y }{\partial x_k}  
=\left(-\frac{A_k}{x-x_k} + \sum^{n}_{j=1}f_{jk} \log(x-x_j)  \right)Y.
\]
\end{theorem}

\begin{proof} The proof is essentially the same as the case of monodromy 
preserving deformations. \cite{Sch05}.

Since 
\begin{eqnarray*}
\frac{d }{dx}Y(x) Y(x)^{-1}&\sim&\frac{d }{dx}Y_j(x) Y_j(x)^{-1}+\frac{L_j}{x-x_j}Y_j(x)Y_j(x)^{-1}\\
&\sim&  Y_j(a_j){L_j}Y_j(x_j)^{-1}\frac{1}{x-a_j}+ O((x-x_j)^0) 
\end{eqnarray*}
for $j=1,2,...,n$, we have 
\begin{equation}\label{gauge}
A_j=Y_j(x_j){L_j}Y_j(x_j)^{-1}.
\end{equation}

Therefore near $x=\infty$, we have the following expansions:
\begin{eqnarray*} 
\frac{\partial Y }{\partial x_k} Y^{-1}
&\sim &\frac{\partial Y_\infty}{\partial x_k} (x) Y_\infty(x)^{-1} -\sum^{n}_{j=1}   f_{jk} \log x,  \\
&\sim & -\sum^{n}_{j=1}   f_{jk} \log x + O(\frac1x) ,
\end{eqnarray*}
since $ \frac{\partial Y_\infty}{\partial x_k} \sim  O(\frac1x)  $ by $Y_\infty(0)=I$. 

In case $k\not= j$, the expansion near  $x=x_j$ is
\begin{eqnarray*}
\frac{\partial Y }{\partial x_k} Y^{-1}
&\sim &\frac{\partial Y_j}{\partial x_k} (x) Y_j(x)^{-1} +   f_{jk} \log(x-x_j), 
\quad (j\not= k). 
\end{eqnarray*}
The expansion near  $x=x_k$ is
\begin{eqnarray*}
\frac{\partial Y }{\partial x_k} Y^{-1}
&\sim &\frac{\partial Y_k}{\partial x_k}(x) Y_k(x)^{-1} 
   -  Y_k(x) \frac{L_k}{x-x_k} Y_k(x)^{-1} +  f_{kk} \log(x-x_k)\\
&\sim &-   \frac{A_k}{x-x_k}   +  f_{kk} \log(x-x_k) +(O(x-x_k)^0).
\end{eqnarray*}
We use \eqref{gauge} to show the last line. Therefore we obtain that
\[
\frac{\partial Y }{\partial x_k} =\left(-\frac{A_k}{x-x_k} 
    + \sum_{j=1}^n  f_{jk} \log(x-x_j)  \right)Y, 
\]
which gives the deformation of $x_k$. \end{proof}

Since the deformation equation contains a logarithmic term 
\[
\nu_k =\sum_{j=1}^n  f_{jk} \log(x-x_j),
\]  the monodromy data is not preserved.  $\nu_k$ satisfies 
\[
\frac {d \nu_k}{dx}=\sum_{j=1}^n \frac{f_{jk}}{x-x_j},
\]
which is essentially equivalent to \eqref{nu} in the work of Chakravarty and Ablowitz. 
It seems difficult to study monodromy evolving deformations when $f_{jk}$'s is 
not scalar functions. 

\section{Halphen's second equation and MED}\label{sec:hal2}

We show that Halphen's second equation is represented by monodromy evolving deformations. 
Let $x_1, x_2, x_3$ be functions of $t$. In case $a+b=c+b=-1/4$, our evolving 
deformation is essentially equivalent to \cite{CA} for DH-V. But ours are simpler 
than their deformation.

We set 
$$Q(x)=x^2 +a (x_1-x_2)^2 +b(x_2-x_3)^2+c(x_3-x_1)^2.$$
Halphen's second equation is 
$$ x_j^{\prime}=Q(x_j), \qquad j=1,2,3.$$

We set 
$$P(x)=(x-x_1)(x-x_2)(x-x_3)$$
and consider the following 2 $\times$ 2 linear system. 
 \begin{eqnarray} 
 \frac {\pa Y}{\pa x} =\left( \frac \mu P + \sum_{j=1}^3 \frac{c_j S}{x-x_j}\right) Y, \label{med1}\\
 \frac {\pa Y}{\pa t} =\left( \nu + \sum_{j=1}^3 c_j x_j S  \right) Y -Q(x) \frac {\pa Y}{\pa x}. \label{med2}
 \end{eqnarray}
Here $\mu$   and  $c_j$'s are constants with $c_1+c_2+c_3=0$, and 
 $S$ is any traceless constant  matrix. 
We assume
$$\frac {\pa\nu}{\pa x} =-\frac {x+x_1+x_2+x_3} P \mu.$$

\begin{theorem}
The compatibility condition of \eqref{med1} and \eqref{med2} gives the Halphen's second equation.
\end{theorem}
We can prove the theorem above directly. 
Therefore  \eqref{med1} and \eqref{med2}  are a Lax pair of Halphen's second equation. 

The local monodromy of $Y_j(x)$ around $x=x_j$ is 
adjoint to $e^{2\pi i L_j}$.  
This deformation does not preserve monodromy data. The local exponent   $L_j$ at $x=x_j$ 
evolves as
\begin{equation*}\label{exp_evol}
\frac{d L_j}{dt}= \frac{2x_j+x_k+x_l }{\prod_{m\not=j} (x_j-x_m)}\mu,
\end{equation*}
where $\{j,k,l\}=\{1,2,3\}$ as a set. 
The singular points $x_j$  also deform as
\begin{equation*}\label{sing_evol}
\frac{d x_j}{dt}= Q(x_j),
\end{equation*}
which is nothing but Halphen's second equation.  

We can eliminate the variables $\mu$ and $\nu$  in \eqref{med1} and \eqref{med2} 
by the rescaling $Y=fZ$  for a scalar function $f=f(x,t)$. 
$f $  satisfies the linear equations 
 \begin{align*}
  \frac {\pa f}{\pa x} &=  \frac \mu P f,\\
 \frac {\pa f}{\pa t} &=  \nu f -Q(x) \frac {\pa f}{\pa x}.
  \end{align*}
The integrability condition for $f$ is 
$$ \frac {\pa P}{\pa t} +Q  \frac {\pa P}{\pa x} -P  \frac {\pa Q}{\pa x} -(x+x_1+x_2+x_3)P=0.$$
And $Z$ satisfies 
\begin{align*}
\frac {\pa Z}{\pa x} &=  \sum_{j=1}^3 \frac{c_j S}{x-x_j}  Z,  \\
\frac {\pa Z}{\pa t} &=\sum_{j=1}^3 c_j x_j S  Z -Q(x) \frac {\pa Z}{\pa x}.  
\end{align*}
The integrability condition for $Z$ gives the sixth Painlev\'e equation. 
In our case, we take the Riccati solution of the sixth Painlev\'e equation,
which reduce to the hypergeometric equation, since the residue matrix of
$(C x+D)/P$ at the infinity is zero. 
  But this hypergeometric equation 
is different from the hypergeometric functions which solve \eqref{hal2}. 

The Halphen's equation is described as MED, but it stands on a similar position  
as the Riccati solution of the  Painlev\'e equations. 
Studies of generic solutions of MED or other special solutions of MED, such as 
algebraic solutions or elliptic solutions are future problems.

\end{document}